\documentclass[a4paper,12pt]{amsart}
\usepackage{etex}
\usepackage{amsmath,amsfonts,amssymb,amsthm}

\usepackage{pstricks}
\usepackage{subfig}
\usepackage{graphicx}
\usepackage{color}
\usepackage[all]{xy}
\usepackage{url}
\usepackage{setspace}
\usepackage{booktabs}
\usepackage{longtable}
\usepackage{latexsym}
\usepackage{tikz}
\usepackage{hyphenat}
\usepackage{hyperref}

\newtheorem{theorem}{Theorem}[section]
\newtheorem{thm}{Theorem}[section]
\newtheorem{lemma}[theorem]{Lemma}
\newtheorem{prop}[theorem]{Proposition}
\newtheorem{proposition}[theorem]{Proposition}

\newtheorem{cor}[theorem]{Corollary}
\newtheorem{conjecture}[theorem]{Conjecture}

\theoremstyle{definition}
\newtheorem{definition}[theorem]{Definition}
\newtheorem{example}[theorem]{Example}

\newtheorem{remark}[theorem]{Remark}

\definecolor{cKlaus}{rgb}{0.1,0.45,0.03}
\definecolor{loesung}{rgb}{0.6,0.10,0.33}
\definecolor{cKlausOK}{rgb}{0.3,0.40,0.33}
\definecolor{intOrange}{rgb}{1.0,.310,.0}

\definecolor{cNathan}{rgb}{0.3,0.40,0.53}

\usepackage[textsize=tiny]{todonotes}
\definecolor{intOrange}{rgb}{1.0,.310,.0} % 255 79 0 International Orange
\newcommand{\cP}{\mathcal{P}}

\newcommand{\cE}{\mathcal{E}}

\newcommand{\bK}{\mathbb{K}}

\newcommand{\ZZ}{\mathbb{Z}}
\newcommand{\bN}{\mathbb{N}}
\newcommand{\NN}{\mathbb{N}}

\newcommand{\KK}{\mathbb{K}}

\newcommand{\bfu}{\mathbf u}
\newcommand{\bfv}{\mathbf v}
\newcommand{\bfw}{\mathbf w}

\newcommand{\PP}{\mathbb P}

\newcommand{\Z}{\mathbb Z}

\newcommand{\Q}{\mathbb Q}
\newcommand{\QQ}{\mathbb Q}
\newcommand{\CA}{{\mathcal A}}

\newcommand{\CO}{{\mathcal O}}

\DeclareMathOperator{\spec}{Spec}

\DeclareMathOperator{\ord}{ord}

\DeclareMathOperator{\conv}{conv}

\DeclareMathOperator{\Div}{{Div}}

\renewcommand{\div}{\operatorname{div}}

\newcommand{\kst}{\,|\;}

\newcommand{\surj}{\rightarrow\hspace{-0.8em}\rightarrow}

\DeclareMathOperator{\Inc}{\nabla_\cP}   % Incarnation/Realization polytope

\DeclareMathOperator{\lin}{lin}
\newcommand{\dfan}{\mathcal{S}}
\newcommand{\marking}{\mathcal{M}}
\newcommand{\xray}{\Sigma^{\times}}
\newcommand{\tv}{\mathbb{TV}}

\begin{document}

\title[Fujita's freeness conjecture for $T$-varieties]
{Fujita's freeness conjecture for $T$-varieties of complexity one}

\author[K.~Altmann]{Klaus Altmann%
% $^{1}$
}
\address{Institut f\"ur Mathematik,
Freie Universit\"at Berlin,
% FU Berlin,
Arnimalle 3,
14195 Berlin,
Germany}
\email{altmann@math.fu-berlin.de}
\author[N.~Ilten]{Nathan Ilten}
\address{
Dept.~of Mathematics,
Simon Fraser University,
8888 University Drive,
Burnaby BC V5A 1S6,
Canada
% Room: SSC K 10517
}
\email{nilten@sfu.ca}
\thanks{Partial support for KA was provided by the
Pacific Institute for the Mathematical Sciences.
Partial support for KA and NI was provided
 by NSERC. We thank Kelli Talaska for helping us track down a reference.}

\begin{abstract}
We prove Fujita's freeness conjecture for Gorenstein complexity-one $T$-varieties with rational singularities.
\end{abstract}

\maketitle

%%%%%%%%%%%%%%%%%%%%%%%%%
%%%   Introduction
%%%%%%%%%%%%%%%%%%%%%%%%%
\section{Introduction}
\label{intro}

In \cite{fujita}, Fujita made the following conjectures:
\begin{conjecture}
	Let $X$ be a projective variety with sufficiently mild singularities, and $H$ an  ample divisor on $X$. 
	\begin{enumerate}
		\item {\emph{Freeness}:} For $m\geq \dim X+1$, $mH+K_X$ is basepoint free.
		\item {\emph{Very ampleness}:} For $m\geq \dim X+2$, $mH+K_X$ is very ample.
	\end{enumerate}
\end{conjecture}
For smooth curves, these conjectures follow from the Riemann-Roch theorem. Both conjectures have been proven for smooth surfaces by Reider \cite{reider:88a}, and the freeness conjecture has been proven for smooth varieties in dimensions three, four, and five by Ein and Lazarsfeld, Kawamata, and Ye and Zhu, respectively \cite{ein:93a,kawamata:97a,ye}.

Weakening the freeness conjecture, one could instead ask only for $mH+K_X$ to be nef for $m\geq \dim X+1$. This has been shown by Ionescu when $X$ is smooth \cite{ionescu:86a}, and by Fujita when $X$ has normal rational Gorenstein singularities \cite{fujita:87a}. This implies the freeness conjecture for Gorenstein toric varieties,  and the very ampleness conjecture for smooth toric varieties. Indeed, for toric varieties, nef is equivalent to basepoint free, and for smooth toric varieties, ample is equivalent to very ample. Payne later showed with a combinatorial argument that the very ampleness conjecture is also true in the Gorenstein case \cite{payne:06a}.

In this paper, we study Fujita's freeness conjecture for \emph{complexity-one $T$-varieties}, a natural generalization of toric varieties. Recall that a complexity-one $T$-variety is a normal variety $X$ equipped with an effective action of an algebraic torus of dimension $\dim X-1$.
Our main result is that the freeness conjecture holds for such varieties:
\begin{theorem}[Freeness]\label{thm:main}
Let $X$ be a projective Gorenstein complexity-one $T$-variety with at worst rational singularities, and $H$ an ample divisor.
		Then for $m\geq \dim X+1$, $mH+K_X$ is basepoint free.
\end{theorem}

There are two key ideas going into the proof of this theorem. 
The first is to use the convexity of a polytope (the so-called \emph{realization polytope}) to provide a lower bound for the degrees of some special divisors on a curve, see Proposition \ref{prop:section} and Corollary \ref{cor:degree}.
The second idea is that for nef divisors on complexity-one $T$-varieties, certain torus fixed points are never part of the base locus, see Proposition \ref{prop:bp}. 

Unlike in the toric case, for divisors on complexity-one $T$-varieties, nef does not imply basepoint free, nor does ample imply very ample. In fact, we show that there is a sequence of smooth $\KK^*$-surfaces $X_k$ with ample divisors $H_k$ such that for $m\leq k$, $m H_k$ is not basepoint free, see Theorem \ref{thm:ex}.

The reason that the freeness conjecture for complexity-one $T$-varieties is so tractable is that these varieties can be described using combinatorial language. We recall the concepts we need in \S \ref{sec:prelim}. In order to better understand the global sections of a $T$-invariant divisor, we introduce the realization polytope in \S \ref{sec:polytope}. After doing some local computations in \S \ref{sec:local}, we prove our main result in \S \ref{sec:fujita}. In \S \ref{sec:nonfree} we conclude by presenting an example of a sequence of smooth $\KK^*$-surfaces with ample divisors such that arbitrarily high multiples are not basepoint free. Throughout this paper, we work over an algebraically closed field $\KK$ of characteristic zero.

\section{Preliminaries on $T$-varieties}\label{sec:prelim}
In this section, we fix notation and recall basic facts about $T$-varieties. For more details on $T$-varieties, see \cite{survey}.

\subsection{Complete Complexity-One $T$-Varieties}
Fix an algebraic to\-rus $T$ with character lattice $M$. Let $N$ be the dual lattice, that is, the lattice of one-parameter subgroups. We denote the associated $\QQ$-vector spaces by $M_\QQ$ and $N_\QQ$.

\begin{definition}
A complexity-one $T$-variety is a normal variety $X$ endowed with an effective $T$-action satisfying $\dim X=\dim T+1$.
\end{definition}

Complete complexity-one $T$-varieties can be encoded via combinatorics:
\begin{definition}\cite{polarized}
A \emph{fansy divisor} $\dfan$ on a smooth projective curve $Y$ is a formal sum 
\[
	\sum_{P\in Y} \dfan_P\otimes P
\]
where $\dfan_P$ are polyhedral complexes whose support is all of $N_\QQ$, sharing a common recession fan $\Sigma$, and such that $\dfan_P\neq \Sigma$ for only finitely many $P$.

When $Y=\PP^1$, a \emph{marking} of $\dfan$ is a subset $\marking$ of the cones of $\Sigma$ such that
\begin{enumerate}
\item If $\tau\prec \sigma$ and $\tau\in \marking$, then $\sigma\in\marking$;
	\item If $\sigma\in \marking$ is full-dimensional,
\[
\deg \dfan^\sigma:=	\sum_{P\in Y} \dfan_P^\sigma  \subsetneq \sigma
\]
where $\dfan_P^\sigma$ denotes the unique polyhedron in $\dfan$ with recession cone $\sigma$;
	\item If $\sigma\in \marking$ is full-dimensional,
		and $\tau\prec \sigma$, then $\tau\in\marking$ if and only if
\[
	(\deg \dfan^\sigma)\cap \tau \neq \emptyset .
\]
\end{enumerate}
	For curves $Y$ of higher genus, a marking is defined similarly, with an additional technical condition that will play no role for us \cite[\S1]{polarized}.
\end{definition}

	In analogy with the correspondence between toric varieties and polyhedral fans, there is a correspondence between complete complexity-one $T$-varieties and marked fansy divisors \cite[\S 1]{polarized}. We will denote the $T$-variety corresponding to a marked fansy divisor $(\dfan,\marking)$ by $\tv(\dfan,\marking)$. Such a variety comes with a rational quotient map $\pi:\tv(\dfan,\marking)\dashrightarrow Y$. The resolution of this rational map is given by the quotient map $\tv(\dfan,\emptyset)\to Y$. 
	
The fixed points of $X=\tv(\dfan,\marking)$ are of two types. A \emph{contraction} fixed point is one contained in the indeterminacy locus of $\pi$. Such fixed points are in bijection to the maximal elements of $\marking$. We call other fixed points \emph{contraction-free}. The contraction-free fixed points contained in a fiber $\pi^{-1}(P)$ are in bijection with the maximal-dimensional polyhedra in $\dfan_P$ whose recession cones are not marked.
This follows from the description of $T$-orbits given in \cite[\S 10]{pdiv}.

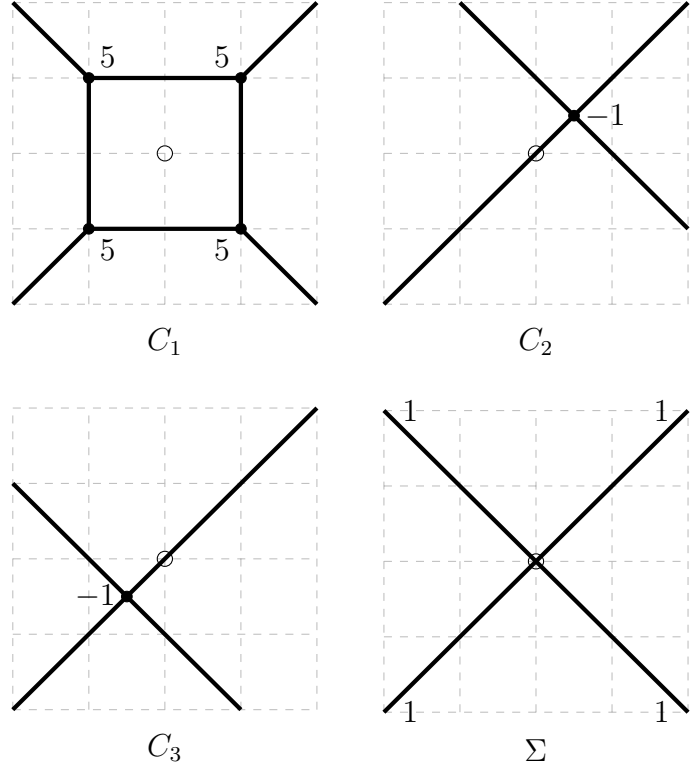
\begin{figure}
	\begin{align*}
	\begin{tikzpicture}
		\draw [lightgray, thin, dashed] (-2,-2) grid (2,2);
		\draw [ultra thick] (-2,-2) -- (-1,-1) -- (-1,1) -- (-2,2);
		\draw [ultra thick] (2,2) -- (1,1) -- (1,-1) -- (2,-2);
		\draw [ultra thick] (1,1) -- (-1,1);
		\draw [ultra thick] (-1,-1) -- (1,-1);
		\draw [fill] (1,1)  circle [radius=.07];
		\draw [fill] (-1,1) circle [radius=.07];
		\draw [fill] (1,-1) circle [radius=.07];
		\draw [fill] (-1,-1) circle [radius=.07];
		\draw  (0,0) circle [radius=.1];
		\node at (0,-2.5) {$C_1$}; 
		\node [below right] at (-1,-1) {$5$};
		\node [below left] at (1,-1) {$5$};
		\node [above right] at (-1,1) {$5$};
		\node [above left] at (1,1) {$5$};
%		\node [below,gray] at (0,0) {$(0,0)-5$};
%		\node [above,gray] at (0,-2) {$(0,1)-4$};
%		\node [below,gray] at (0,2) {$(0,-1)-4$};
%		\node [right,gray] at (1,0) {$(-1,0)-4$};
	\end{tikzpicture}\qquad
	\begin{tikzpicture}
		\draw [lightgray, thin, dashed] (-2,-2) grid (2,2);
		\draw [ultra thick] (-1,2) -- (2,-1);
		\draw [ultra thick] (-2,-2) -- (2,2);
		\draw [fill] (.5,.5)  circle [radius=.07];
		\draw  (0,0) circle [radius=.1];
		\node at (0,-2.5) {$C_2$}; 
		\node [right] at (.5,.5) {$-1$};
	\end{tikzpicture}
	\\
	\begin{tikzpicture}
		\draw [lightgray, thin, dashed] (-2,-2) grid (2,2);
		\draw [ultra thick] (1,-2) -- (-2,1);
		\draw [ultra thick] (-2,-2) -- (2,2);
		\draw [fill] (-.5,-.5)  circle [radius=.07];
		\draw  (0,0) circle [radius=.1];
		\node at (0,-2.5) {$C_3$}; 
		\node [left] at (-.5,-.5) {$-1$};
	\end{tikzpicture}\qquad
	\begin{tikzpicture}
		\draw [lightgray, thin, dashed] (-2,-2) grid (2,2);
		\draw [ultra thick] (2,-2) -- (-2,2);
		\draw [ultra thick] (-2,-2) -- (2,2);
		\draw  (0,0) circle [radius=.1];
		\node at (0,-2.5) {$\Sigma$}; 
		\node [right] at (-1.9,2) {$1$};
		\node [right] at (-1.9,-2) {$1$};
		\node [left] at (1.9,2) {$1$};
		\node [left] at (1.9,-2) {$1$};
	\end{tikzpicture}
	\end{align*}
	\caption{A singular $(\KK^*)^2$ threefold}\label{fig:ex1}
\end{figure}

\begin{example}[A singular $(\KK^*)^2$ threefold]\label{ex1}
	Consider the polyhedral subdivisions $C_1,C_2,C_3,\Sigma$ of $\QQ^2$ pictured in Figure \ref{fig:ex1}; the origin is marked in each subdivision by a hollow circle. We construct a fansy divisor on $\PP^1$ by
	\[
\dfan=C_1\otimes P_0+C_2\otimes (P_1+P_2+P_3)+C_3\otimes(P_4+P_5+P_6)
	\]
	where $P_0,\ldots,P_6$ are distinct points on $\PP^1$. This fansy divisor has $\Sigma$ as its recession fan. The corresponding variety $\tv(\dfan,\emptyset)$ is a threefold with Gorenstein singularities.
\end{example}

\subsection{Invariant Divisors}\label{sec:divisors}
Fix a marked fansy divisor $(\dfan,\marking)$ on $Y$ with recession fan $\Sigma$. Let $\xray$ be the set of those rays $\rho\in \Sigma$ which are not in $\marking$.
Following \cite{divisors,polarized}, invariant prime divisors on $X=\tv(\dfan,\marking)$ come in two types. \emph{Vertical divisors} correspond to pairs $(P,v)$, where $P\in Y$ and $v$ is a vertex of $\dfan_P$; we denote the corresponding divisor by $D_{P,v}$. \emph{Horizontal divisors} correspond to rays $\rho\in \xray$; we denote the corresponding divisor by $D_\rho$. 

Semi-invariant rational functions on $X$ are elements of $\KK(Y)[M]$.
\begin{lemma}[{\cite[Proposition 3.14]{divisors}}]\label{lemma:haupt}
	The principal divisor associated to $f\cdot \chi^u\in \KK(Y)[M]$ is 
	\[
		\sum_{\rho\in\xray} \langle \rho,u\rangle D_\rho+\sum_{{P,v}} \mu(v)(\langle u,v\rangle +\ord_P f) D_{P,v}
	\]
	where $\ord_P f$ is the order of vanishing of $f$ at $P$, $\mu(v)$ is the smallest natural number such that $\mu(v)v\in M$, and by abuse of notation, $\rho$ denotes both a ray and its primitive lattice generator.
\end{lemma}

Invariant \emph{Cartier} divisors on $\tv(\dfan,\marking)$
may be encoded by a collection $h=(h_P)$ of piecewise affine functions $h_P:|\dfan_P| \to \QQ$ satisfying some additional properties, see
\cite[Definition 3.4, 3.8 ]{divisors}.
In particular, the \emph{linear part}
\[
h_P^{\lin}(v):=\lim_{k\to \infty} \frac{h(kv)}{v}
\]
of $h_P$ is some piecewise linear function $h^{\lin}$ independent of $P$. It is linear on the cones of the recession fan $\Sigma$.
We denote the divisor corresponding to $h=(h_P)$ by $D_h$, and we have the following formula \cite[Corollary 3.19]{divisors}:
\[
	D_h=-\sum_{\rho\in\xray} h^{\lin}(\rho) D_\rho-\sum_{{P,v}} \mu(v)h_P(v) D_{P,v}.
\]
Such a collection of functions $h=(h_P)$ is called a \emph{Cartier support function}.

Let $h$ be a Cartier support function. On each maximal cell $\Delta$ of $\dfan_P$, $(h_P)_{|\Delta}$ is affine linear of the form 
\[
	(h_P)_{|\Delta}(v)=\langle v,u\rangle +a_{P}
\]
for some $u\in M$, $a_P\in\ZZ$. This determines local generators for $\CO(D_h)$ as follows. If $\Delta$ corresponds to a contraction-free fixed point $x$, a local generator for $\CO(D_h)$ at $x$ is given by  $f\cdot \chi^u$, where $f\in \KK(Y)$ is any function satisfying $\ord_P f=a_P$. For a contraction fixed point $x$ corresponding to a full-dimensional $\sigma\in\marking$, a local generator for $\CO(D_h)$ at $x$  is $f\cdot \chi^u$
with  $\ord_P f=a_P$ for all $P$, with $u$ and $a_P$ determined as above by restricting $h_P$ to $\dfan_P^\sigma$.

Given a Cartier support function $h$, we set
\[
\square=\square_h:=\{u\in M_\Q\kst h^{\lin}(v)\leq \langle v,u\rangle
\ \forall v\in N_\Q\}.
\]
On this polytope we have the piecewise affine concave 
functions
\begin{align*}
	h_P^*:&\Box\to \QQ\\
	&u\mapsto \min_{v\in \dfan_P{(0)}}
\big(\langle v,u\rangle -h_P(v)\big)
\end{align*}
for $P\in Y$.
Here, $\dfan_P(0)$ denotes the vertices of $\dfan_P$.
Set $h^*(u)=\sum_P h_P^*(u)\cdot P\in\Div_\QQ Y$.

\begin{lemma}[{\cite[Proposition 3.23]{divisors}}]\label{lemma:gs}
For a Cartier divisor $D_h$, there is a graded isomorphism
\[
	H^0(X,\CO(D_h))\cong \bigoplus_{u\in \Box_h\cap M} H^0(Y,\CO(\lfloor h^*(u)\rfloor ))\cdot \chi^u.
\]
\end{lemma}

\begin{remark}
	Lemma \ref{lemma:gs} can be rephrased as follows: for $f\in \KK(Y)$ and $u\in M$, $f\cdot \chi^u$ is a global section of $\CO(D_h)$ if and only if $u\in \Box_h$ and 
	\[
-\ord_P f\leq h_P^*(u)
	\]
	for all $P\in Y$.
\end{remark}

When taking multiples of a divisor $D_h$, $h$ and $h^*$ behave as expected: for $m\in \NN$,
\begin{align*}
	m\cdot D_h&=D_{mh}\\
	\Box_{mh}&=m\cdot\Box_h\\
	(mh)^*(u)&=m\cdot h^*(u/m).
\end{align*}

\begin{example}\label{ex2}
	We continue Example \ref{ex1}. We consider the Weil divisor $H$ on $X=\tv(\dfan,\emptyset)$ where the coefficients for the prime divisors $D_{P,v}$ and $D_\rho$ are recorded in Figure \ref{fig:ex1}. This divisor is Cartier; the piecewise linear functions $h_P$ are described on each region of linearity in Figure~\ref{fig2} as $h_P(v)=\langle v,u\rangle +a_P$. The polytope $\Box_h$ and the functions $h_P^*$ are described in Figure \ref{fig3}; the subdivisions of $\Box_h$ indicate the regions of linearity of $h_P^*$, with values recorded at the vertices.

	In particular, 
	\[
		h^*((0,0))=5P_0-\frac{1}{2}(P_1+P_2+\ldots+P_6)
	\]
	which implies by Lemma \ref{lemma:gs} that $H$ has no global sections of the form $f\cdot \chi^{(0,0)}$ with $f\in \KK(\PP^1)$.
\end{example}

\begin{figure}
	\begin{align*}
	\begin{tikzpicture}
		\draw [ultra thick] (-2,-2) -- (-1,-1) -- (-1,1) -- (-2,2);
		\draw [ultra thick] (2,2) -- (1,1) -- (1,-1) -- (2,-2);
		\draw [ultra thick] (1,1) -- (-1,1);
		\draw [ultra thick] (-1,-1) -- (1,-1);
		\draw  (0,0) circle [radius=.1];
		\node [above,gray] at (0,0) {$u=(0,0)$};
		\node [below,gray] at (0,0) {$a_P=-5$};
		\node [above,gray] at (0,-2) {$u=(0,1)$};
		\node [below,gray] at (0,-2) {$a_p=-4$};
		\node [above,gray] at (0,2) {$u=(0,-1)$};
		\node [below,gray] at (0,2) {$a_P=-4$};
		\node [above right,gray] at (1,0) {$u=(-1,0)$};
		\node [below right,gray] at (1,0) {$a_P=-4$};
		\node [above left,gray] at (-1,0) {$u=(1,0)$};
		\node [below left,gray] at (-1,0) {$a_P=-4$};
	\end{tikzpicture}\qquad
	\begin{tikzpicture}
		\draw [ultra thick] (-1,2) -- (2,-1);
		\draw [ultra thick] (-2,-2) -- (2,2);
		\draw  (0,0) circle [radius=.1];
		\node [above,gray] at (0,-2) {$u=(0,1)$};
		\node [below,gray] at (0,-2) {$a_p=0$};
		\node [above,gray] at (.5,2) {$u=(0,-1)$};
		\node [below,gray] at (.5,2) {$a_P=1$};
		\node [above right,gray] at (1,.5) {$u=(-1,0)$};
		\node [below right,gray] at (1,.5) {$a_P=1$};
		\node [above left,gray] at (-1,0) {$u=(1,0)$};
		\node [below left,gray] at (-1,0) {$a_P=0$};
	\end{tikzpicture}
	\\
	\\
	\begin{tikzpicture}
		\draw [ultra thick] (1,-2) -- (-2,1);
		\draw [ultra thick] (-2,-2) -- (2,2);
		\draw  (0,0) circle [radius=.1];
		\node [above,gray] at (-.5,-2) {$u=(0,1)$};
		\node [below,gray] at (-.5,-2) {$a_p=1$};
		\node [above,gray] at (0,2) {$u=(0,-1)$};
		\node [below,gray] at (0,2) {$a_P=0$};
		\node [above right,gray] at (1,0) {$u=(-1,0)$};
		\node [below right,gray] at (1,0) {$a_P=0$};
		\node [above left,gray] at (-1,-.5) {$u=(1,0)$};
		\node [below left,gray] at (-1,-.5) {$a_P=1$};
	\end{tikzpicture}\qquad
	\begin{tikzpicture}
		\draw [ultra thick] (2,-2) -- (-2,2);
		\draw [ultra thick] (-2,-2) -- (2,2);
		\draw  (0,0) circle [radius=.1];
		\node [above,gray] at (0,-2) {$u=(0,1)$};
		\node [below,gray] at (0,-2) {$a_p=0$};
		\node [above,gray] at (0,2) {$u=(0,-1)$};
		\node [below,gray] at (0,2) {$a_P=0$};
		\node [above right,gray] at (1,0) {$u=(-1,0)$};
		\node [below right,gray] at (1,0) {$a_P=0$};
		\node [above left,gray] at (-1,0) {$u=(1,0)$};
		\node [below left,gray] at (-1,0) {$a_P=0$};
	\end{tikzpicture}
	\end{align*}
	\caption{A Cartier support function}\label{fig2}

	\vspace{1cm}

	\begin{tikzpicture}
	\draw [lightgray, thin, dashed] (-2,-2) grid (2,2);
	\draw [ultra thick] (0,-1) -- (-1,0) -- (0,1) -- (1,0) -- (0,-1);
	\draw [ultra thick,dashed] (0,-1) --  (0,1);
	\draw [ultra thick,dashed] (-1,0) --  (1,0);
	\node [left] at (-1,0) {$4$};
	\node [right] at (1,0) {$4$};
	\node [above] at (0,1) {$4$};
	\node [below] at (0,-1) {$4$};
	\node [below right] at (0,0) {$5$};
	\node at (0,-2.3) {$P=P_0$};
	\end{tikzpicture}
	\begin{tikzpicture}
	\draw [lightgray, thin, dashed] (-2,-2) grid (2,2);
	\draw [ultra thick] (0,-1) -- (-1,0) -- (0,1) -- (1,0) -- (0,-1);
	\node [left] at (-1,0) {$-1$};
	\node [right] at (1,0) {$0$};
	\node [above] at (0,1) {$0$};
	\node [below] at (0,-1) {$-1$};
	\node at (0,-2.3) {$P=P_1,P_2,P_3$};
	\end{tikzpicture}
	\begin{tikzpicture}
	\draw [lightgray, thin, dashed] (-2,-2) grid (2,2);
	\draw [ultra thick] (0,-1) -- (-1,0) -- (0,1) -- (1,0) -- (0,-1);
	\node [left] at (-1,0) {$0$};
	\node [right] at (1,0) {$-1$};
	\node [above] at (0,1) {$-1$};
	\node [below] at (0,-1) {$0$};
	\node at (0,-2.3) {$P=P_4,P_5,P_6$};
	\end{tikzpicture}

	\caption{$\Box_h$ and $h^*$}\label{fig3}
\end{figure}

\subsection{Positivity and adjunction}\label{sec:adjunction}

When $D_h$ is nef, each $h_P$ is a concave function on $\dfan_P$ \cite[Corollary 3.29]{divisors}. This implies that each $(u,-a_P)$ determined from a maximal cell $\Delta$ of $\dfan_P$ above is a vertex of the graph of $h_P^*$, and these are the only vertices.
In particular, $\Box_h$ is a lattice polytope, and the vertices of the graph of $h_P^*$ are all integral.
If $x$ is a contraction-free fixed point of $X$ in $\pi^{-1}(P)$, there is a corresponding vertex $(u,h_P^*(u))$ of the graph of $h_P^*$ with  a local generator for $\CO(D_h)$ at $x$ is given by $f\cdot \chi^u$, where $\ord_P f=-h_P^*(u)$. On the other hand, if $x$ is a contraction fixed point of $X$, the corresponding $u$ is actually a vertex of $\Box_h$, and a local generator for $\CO(D_h)$ at $x$ is given by $f\cdot \chi^u$, where $\ord_P f=-h_P^*(u)$ for all $P\in Y$.
 See \cite[\S 3]{polarized} for details.

When $D_h$ is ample, $h_P$ is strictly concave on $\dfan_P$, 
and we obtain a bijection between maximal cells in $\dfan_P$ and vertices of the graph of $h_P^*$.
In this situation, contraction fixed points correspond to vertices $u$ of $\Box_h$ with $h^*(u)$ a principal divisor, and non-contraction fixed points in $\pi^{-1}(P)$ correspond to vertices $(u,h_P^*(u))$ of the graph of $h_P^*(u)$ satisfying $\deg h^*(u) >0$.

\begin{prop}\label{prop:bp}
Let $D$ be a nef $T$-invariant Cartier divisor on a complete complexity-one $T$-variety $X$. Then the base locus of $D$ contains no contraction fixed points.
\end{prop}
\begin{proof}
Fixing notation, let $h$ be such that $D=D_h$. Consider a contraction fixed point $x\in X$; this corresponds to a maximal cone $\sigma\in \marking$. 
As noted above, a local generator for $\CO(D_h)$ at $x$ is given by $f\cdot \chi^u$, where $\ord_P f=-h_P^*(u)$ for all $P\in Y$ and $u$ is the  vertex of $\Box_h$ corresponding to $x$. But then $f\cdot \chi^u$ is a global section of $\CO(D_h)$ by Lemma \ref{lemma:gs}.
\end{proof}

\begin{lemma}\label{lemma:bp}
Let $D_h$ be a nef $T$-invariant Cartier divisor on $X=\tv(\dfan,\marking)$. Then 
$D_h$ is basepoint free if and only if for every $(u,h_P^*(u))$ corresponding to a non-contraction fixed point $x$ of $X$, the linear system $|\lfloor h^*(u)\rfloor |$ is non-empty with no basepoint at $P$.
\end{lemma}
\begin{proof}
	From the discussion above, $x$ is not a basepoint of $D_h$ if and only if $f$ is a global section of $\CO(\lfloor h^*(u)\rfloor )$, where $\ord_P f=-h_P^*(u)$. But this is equivalent to $|\lfloor h^*(u)\rfloor |$ being non-empty with no basepoint at $P$. By Proposition \ref{prop:bp}, we already know that no contraction fixed points are base points. But if the base locus of $D_h$ contains no torus fixed points, then it must be empty.
\end{proof}

We will need a description of a canonical divisor on $X$. Fix a canonical divisor
\[
K_Y=\sum_{P\in Y} a_P\cdot P
\]
on $Y$ and set $b_P=a_P+1$.
Then a canonical divisor on $X$ is given by
\begin{equation}\label{eqn:canonical}
K_X=-\sum_{\rho\in \xray} D_\rho
+\sum_{P,v} \big(\mu(v)\cdot b_P-1\big)\cdot D_{(P,v)},
\end{equation}
see \cite[Theorem 3.21]{divisors}.
Denote the interior of $\Box_h$ by $\Box_h^\circ$.
\begin{prop}\label{prop:canonical}
	Let $D_h$ be an ample Cartier divisor on $X$. Then 
	$f\cdot \chi^u\in \KK(Y)[M]$ is a global section of $\CO(D_h+K_X)$
	if $u\in \Box_h^\circ\cap M$ and 
\[
-\ord_Pf<h_P^*(u)+b_P
\]
for all $P\in Y$.
\end{prop}
\begin{proof}
	Consider any $f\cdot \chi^u\in \KK(Y)[M]$. Then by Lemma \ref{lemma:haupt} and the formula for $K_X$, 
\begin{align*}
\div(f\chi^u)+H+K_X &=
\sum_{\rho\in \xray} \big(\langle \rho, u\rangle -
h^{\lin}(\rho)-1\big)\cdot D_\rho\;+
\\
& \hspace{-3em}+\sum_{(P,\,v)}\mu(v)\,
\big(\langle v,u\rangle + \ord_Pf -h_P(v) + b_P
-\frac{1}{\mu(v)}\big) \cdot D_{P,v}.
\end{align*}

It follows that $\div(f\chi^u)+H+K_X\geq 0$ if and only if
$
\,\langle \rho,u\rangle > h^{\lin}(\rho)\,
$ 
for $\rho\in \xray$ and
$$
\langle u,v\rangle + \ord_Pf -h_P(v) + b_P >0.
$$
The first condition is certainly satisfied if $u\in \Box_h^\circ$. 

The second condition is supposed to be fulfilled for all
	$v\in\dfan_P(0)$. Hence, it may be written as
$$
h_P^*(u) + \ord_Pf + b_P =
\min_{v\in\dfan_P(0)}\big(\langle u,v\rangle -h_P(v)\big)
+ \ord_Pf + b_P >0,
$$
and this just means $\,-\ord_Pf < {h}_P^*(u)+b_P$.
The proposition follows.
\end{proof}
\begin{remark}
	We note that Proposition \ref{prop:canonical} only gives a \emph{sufficient} condition for $f\cdot\chi^u$ to be a global section of $\CO(D_h+K_X)$. The proposition only gives an exact characterization of those sections whose degree lies in $\Box_h^\circ$.
\end{remark}

\begin{example}
	We continue Examples \ref{ex1} and \ref{ex2}. The functions $h_P$ are strictly concave, and $\deg h^*(u)>0$ at the vertices of $\Box_h$. The necessary criteria for a $T$-invariant divisor $D_h$ to be ample given in \S\ref{sec:adjunction} are actually sufficient; see \cite[Corollary 3.28]{divisors} for a precise statement. In any case, in this example, the divisor $H=D_h$ is ample. Nonetheless, $H$ has a basepoint at the torus fixed point $x$ corresponding to the compact cell in $\dfan_0$, or equivalently, the vertex over $(0,0)$ in the graph of $h_{P_0}^*$. Indeed, a local generator for $\CO(H)$ at $x$ must have degree $(0,0)$, see Lemma \ref{lemma:bp}. However, $\CO(H)$ has no global sections of that degree as noted previously.

On the other hand, $\CO(2H)$ does have a global section of degree $(0,0)$, and it follows that $2H$ is basepoint free.
\end{example}

\section{The Realization Polytope}\label{sec:polytope}
Let $\Box\subset M_\QQ$ be a polytope, and 
\[
	\Psi=\sum_{P\in Y} \Psi_P\otimes P,
\]
where each $\Psi_P$ is a piecewise affine concave function $\Psi_P:\Box\to \QQ$, $\deg \Psi(u)>0$ for $u\in \Box^\circ$, and only finitely many $\Psi_P$ are not the constant zero function. The situation we will be primarily interested in is when $\Psi=h^*$, for $h$ the Cartier support function encoding some ample divisor $H=D_h$ on $X$.

\subsection{Global sections}
We are interested in visualizing global sections of the divisor $\Psi(u)$ on $Y$. If $Y$ is not rational, this becomes more subtle, so instead we will initially settle for describing degree zero divisors $F$ on $Y$ such that for fixed $u\in \Box$,
\[
F+\Psi(u)\geq 0.
\]

Let $\cP\subset Y$ be some subset of cardinality $r$ which includes all those $P$ for which $\Psi_P$ is non-trivial.
We denote 
\[
\Q^{\cP}_0:=
\left\{x\in\Q^{\cP}\kst \sum_P x_P=0\right\}.
\]
This is isomorphic to $\QQ^{r-1}$.
Similarly we define the integral version $\Z^{\cP}_0\cong\Z^{r-1}$.

\begin{definition}
Given $\Psi$ as above, we define its \emph{realization} with respect to $\cP$ as the polytope
	\begin{align*}
		\Inc(\Psi)=\left\{(u,x)\in \square\times\Q^{\cP}_0\kst
x_P\leq \Psi_P^*(u)\;\mbox{for } P\in\cP\right\}
\subset M_\Q\times\Q^{\cP}_0.
\end{align*}
\end{definition}
\begin{remark}
	Suppose that $X$ is rational, and $\Psi=h^*$ for $h$ corresponding to an ample divisor $H=D_h$. Then $\Inc(\Psi)$ is just the momentum polytope of the ambient 
toric variety from the well poised embedding of $X$ with respect to $H$
obtained in \cite{wellPoised} by using the Cayley trick, see in particular \cite[\S2.4]{wellPoised}.
\end{remark}
\begin{remark}
We can recover $\Psi$, that is, all functions
$\Psi_P:\square\to\QQ$, from $\Inc(\Psi)$.
\end{remark}

The polytope $\Inc(\Psi)$ comes with a natural projection \[p:\Inc(\Psi)\to \Box.\] We set
\[
	\Inc(\Psi)_u=p^{-1}(u).
\]
This is always an $(r-1)$-dimensional simplex with dilation factor of $\deg \Psi(u)$.

\begin{prop}\label{prop:inc}
For any $u\in \Box\cap M$, there is a bijection between lattice points of $\Inc(\Psi)_u$ and degree zero divisors $F$ with support contained in $\cP$ satisfying
\[F+\Psi(u)\geq 0.\]
Similarly, there is a bijection between interior lattice points of $\Inc(\Psi)_u$ and degree zero divisors $F=\sum c_PP$ with support contained in $\cP$ satisfying
\[c_P+\Psi(u)_P>0\ \forall P\in\cP.\]
\end{prop}
\begin{proof}
	The map $\ZZ_0^\cP\to \Div Y$ sending $(c_P)$ to $\sum_{P\in \cP} c_P\cdot P$ induces the desired bijection.
\end{proof}

\begin{remark}
	If one is interested only in those $F+\Psi(u)$ which are actually in the linear system $|\lfloor \Psi(u)\rfloor |$, we must eliminate those lattice points $(c_P)$ from $\ZZ_0^\cP$ for which $\sum c_P\cdot P$ is not principal; what remains is a sublattice $L$ of ${\ZZ}_0^\cP$. Note that if $Y=\PP^1$, we don't have to change anything.

Those $F$ which are principal determine a global section of $\CO(\Psi(u))$ up to scaling by $\KK^*$. By possibly enlarging $\cP$, we eventually arrive at a spanning set of global sections for $\CO(\Psi(u))$. Note that if $Y=\PP^1$, any set $\cP$ as above with at least two elements will suffice. In light of Lemma \ref{lemma:gs}, $\Inc(h^*)$ is a polytope whose lattice points (with respect to $L$)  give a spanning set for the global sections of $\CO(D_h)$.
\end{remark}

\subsection{The faces of $\Inc(\Psi)$}
The following discussion of the faces of $\Inc(\Psi)$ is not required in the rest of the paper, but is useful for better understanding the structure of the realization polytope.

The projection $p:\Inc(\Psi)\to \Box$ is a special instance of
a projection $p:\nabla \surj \Box $ among two polytopes as it was considered
in \cite[Theorem 3.1]{fiberPolytopes}, see also  \cite[Section 2]{reinerBaues}. The relevant result is a 1-1-correspondence between the so-called
$p$-coherent polyhedral subdivisions of $\Box$, on the one hand, and the
faces of the fiber polytope $\Sigma(\nabla \stackrel{p}{\to}\Box)$ of $p$,
on the other.
The latter is the average fiber (with respect to Minkowski addition).
This is especially easy in our case: it is just a dilation of the
standard simplex~$\Delta$. 

The main idea of this correspondence is, in our
setting, that every pair $(u,f)$ consisting of a point $u\in\square$ and a
face $f\preceq\Delta\cong\Inc(\Psi)_u$ gives rise to a unique face $F\preceq\Inc(\Psi)$
such that $\Inc(\Psi)_u$ intersects its relative interior. 
These ideas lead to a description of the face structure of $\Inc$.

\begin{definition}
\label{def-subdivAI}
For each $I\subseteq\cP$ we denote by $\CA_I$ 
the polyhedral subdivision of $\square$ obtained as the regions of
(affine) linearity of the function 
\[\sum_{P\in I} \Psi_P.\]
\end{definition}

\noindent Note $I=\emptyset$ corresponds to the undivided $\square$, i.e.\
$\CA_\emptyset$ is just the face lattice of $\square$.

\begin{proposition}\label{prop:faces}
The faces of $\Inc(\Psi)$ correspond to pairs
$(I,A)$ with $I\subsetneq\cP$ and $A\in\CA_I$
% For each $(I,A)$, 
where the associated face is
$$
F(I,A):=\{(u,x)\in\Inc(\Psi)\kst
u\in A \mbox{\rm\ and } x_P=\Psi_P(u)\mbox{\rm\ for } P\in I\}.
$$
All pairs $(I,A)$ and $(I',A)$ will be identified whenever
$\deg(\Psi_{|A})=0$.
\end{proposition}
\begin{proof}
	This is a special case of \cite[Theorem 3.1]{fiberPolytopes}.
\end{proof}

\begin{remark}[Facets and divisors]
Facets of $\Inc(\Psi)$ correspond to either facets $A\prec \Box$ on which $\deg \Psi_{|A}\neq 0$ (via the choice $I=\emptyset$), or to facets of the graph of $\Psi_P$ for some $P\in\cP$ (via the choice $I=\{P\}$).

When $\Psi=h^*$ for some ample divisor $D_h$ on $X=\tv(\dfan,\marking)$, facets of $\Inc(\Psi)$ are in bijection with $T$-invariant prime divisors of $X$. A facet $A\prec \Box$ corresponds to $D_\rho$, where $\rho$ is the inward normal vector to $A$; the condition $\deg \Psi_{|A}\neq 0$ ensures that $\rho\in\xray$. A facet of the graph of $\Psi_P$ corresponds to $D_{P,v}$, where $(v,-1)$ is the normal vector to this facet.
\end{remark}

\begin{remark}[Torus fixed points]
	When $\Psi=h^*$ for some ample divisor $D_h$ on $X=\tv(\dfan,\marking)$, the torus fixed points of $X$ can be identified with special faces of $\Inc(\Psi)$ via Proposition \ref{prop:faces}. Indeed, contraction fixed points of $X$ correspond to vertices $u$ of $\Box$ with $\deg \Psi(u)=0$. These can be thought of as vertices of $\Inc(\Psi)$ which are the only element in the respective fiber over $\Box$. 

Contraction-free fixed points of $X$ correspond to vertices $(u,\Psi_P(u))$ of the graph of some $\Psi_P$. These correspond to faces of $\Inc(\Psi)$ which are also facets of $\Inc(\Psi)_u$ for some $u\in \Box$. 
\end{remark}

\section{Local Calculations}\label{sec:local}
In this section we let $X=\tv(\dfan,\marking)$, and take $H=D_h$ to be an ample divisor on $X$. We use the canonical divisor $K_X$ as described in equation \eqref{eqn:canonical}. 

For any vertex $\bfv=(v,h_P^*(v))$
of the graph of $h_P^*$ with $\deg h^*(v)>0$, we let $C_\bfv$ be the cone in $M_\QQ\times \QQ$ generated by 
\[
(u,k)-\bfv,\qquad u\in \Box,\quad k\leq h_P^*(u).
\]
Let $t$ be a local parameter of $Y$ at $P$, and let $x\in X$ be the contraction-free fixed point corresponding to $\bfv$. 
An argument similar to that used in \cite[Proposition 2.6]{singularities} shows that sending $\chi^{(0,-1)}$ to $t$ induces an isomorphism between an \'etale neighborhood of $x$ in $X$
and 
\[
	U_{\bfv}=	\spec \KK[C_\bfv\cap (M\times \ZZ)].
	\]
Let 
\[\bfw_1=(w_1,\alpha_1),\ldots,\bfw_n=(w_n,\alpha_n)\in M\times \ZZ\]
denote the primitive generators of the rays of $C_\bfv$.

\begin{lemma}\label{lemma:gorenstein}
Assume that $X$ is Gorenstein at $x$. A local generator for $\CO(K_X)$ at $x$ can be represented by a function of the form  $f\chi^u$, where
\begin{align*}
u&=\sum \lambda_i w_i\\
\ord_P f&=-\alpha-b_P\qquad \alpha=\sum \lambda_i \alpha_i
\end{align*}
for some $0\leq \lambda_i\leq 1$ with the set
\[
\{
\bfw_i\ |\ \lambda_i\neq 0
\}
\]
linearly independent.
Furthermore, $\alpha$ is the largest integer such that for the $u$ given above,
\[
\bfu=(u,\alpha)\in C_\bfv^\circ.
\]
	\end{lemma}
	\begin{proof}
		Under the \'etale map identifying $(X,x)$ and $U_{\bfv}$, $K_X$ maps to the divisor
		\[
			-B+\div (\chi^{(0,-b_P)}),
		\]
		where $B$ is the toric boundary of $U_{\bfv}$. Furthermore,
		\[
			B=\div \chi^{\bfu},
		\]
	with $\bfu$ the unique interior lattice point of $C_\bfv$ such that 
\begin{equation}\label{eqn:u}
C_\bfv^\circ\cap (M\times \ZZ)=\bfu+(C_\bfv\cap(M\times \ZZ)).
\end{equation}
When $C_\bfv$ is smooth, $\bfu=\sum \bfw_i$. When $C_\bfv$ is simplicial (but not smooth), $\bfu=\sum \lambda_i\bfw_i$ for some $0<\lambda_i\leq 1$ with $\sum \lambda_i <\dim X$. For arbitrary $C_\bfv$, we may triangulate $C_\bfv$ to reduce to the simplicial case, and obtain again $\bfu=\sum \lambda_i\bfw_i$ for some $0\leq \lambda_i\leq 1$ with 
the set 
$\{
\bfw_i\ |\ \lambda_i\neq 0
\}$
linearly independent.
The first claim now follows by pulling back along the \'etale map.
For the second claim, we use the characterization of $\bfu$ (and thus $\alpha$) by \eqref{eqn:u}, and the fact that $(0,-1)\in C_\bfv$.
\end{proof}

\begin{figure}
	\begin{tikzpicture}
	\draw [lightgray, fill=lightgray] (0,2) -- (2,2) -- (4,0) -- (5,-1) -- (0,-1) -- (0,2);
	\draw [gray, thin, dashed] (0,-1) grid (4,3);
	\draw [ultra thick] (0,1) -- (1,2) -- (2,2) -- (3,1) -- (4,-1); 
	\draw [fill] (1,2)  circle [radius=.07];
		\draw [fill] (2,2)  circle [radius=.07];
		\draw [fill] (3,1)  circle [radius=.07];
		\draw [fill] (2,1)  circle [radius=.07];
		\node [above right] at (2,2) {$\bfv$};
		\node [above] at (1,2) {$\bfw_1+\bfv$};
		\node [above right] at (3,1) {$\bfw_2+\bfv$};
		\node [below] at (2,1) {$\bfu+\bfv$};
	\end{tikzpicture}

	\caption{A local generator for $\CO(K_X)$}\label{fig:gorenstein}
\end{figure}

\begin{example}
	Consider Figure \ref{fig:gorenstein}. The black line depicts the graph of some $h_P^*$, with $\bfv=(v,h_P^*(v))$ a vertex of this graph. The gray region is the translate $\bfv+C_\bfv$, and $\bfw_1,\bfw_2$, and $\bfu$ are as depicted.
In this example, $\bfu=\bfw_1+\bfw_2$  ($C_\bfv$ is a smooth cone), $u=v$, and $\alpha=-1$.
\end{example}

\section{Freeness}\label{sec:fujita}
Let $X=\tv(\dfan,\marking)$ be a Gorenstein complexity-one $T$-variety of dimension $d$, and $H=D_h$ an ample $T$-invariant Cartier divisor. 
As in \S\ref{sec:adjunction}, we have fixed a canonical divisor $K_Y=\sum a_P P$ on $Y$, which gives rise to a canonical divisor $K_X$ on $X$. We let $E=\sum e_P P$ be any divisor of degree $2g$ on $Y$, where $g$ is the genus of $Y$.
We take $\cP$ to be any subset of $Y$ containing all $P$ for which $h_P^*$ is non-trivial, and those in the support of $K_Y,E$.
Set
\[
b_P'=b_P-e_P=a_P+1-e_P
\]
for $p\in\cP$.
It follows that $\sum b_P' =r-2$.

Let $\bfv=(v,h_P^*(v))$ be a vertex of the graph of $h_P^*$ corresponding to a contraction-free fixed point $x$.
We let $(u,\alpha)$, $(w_i,\alpha_i)$ be as in the statement of Lemma \ref{lemma:gorenstein}.

\begin{prop}\label{prop:section}
Fix $m\geq d+1$ and set $\Psi=(mh)^*+\sum b_P'\otimes P$.
\begin{enumerate}
	\item\label{item1} The lattice point $u+mv$ is in the interior of $m\cdot \Box_h$.
	\item\label{item2} There is a lattice point in the interior of $\Inc(\Psi)_{u+mv}$.
\end{enumerate}
\end{prop}
\begin{proof}
We first claim that there is some $\varepsilon>0$ such that $\varepsilon u+mv$ is in the interior of $m\cdot \Box_h$. If $mv$ is in the interior of $m\cdot \Box_h$ this is immediate. For the case when $mv$ is in the boundary of $m\cdot \Box_h$, we argue as follows.
		By Lemma \ref{lemma:gorenstein}, $u$ is in the interior of the projection $\overline C$  of $C_\bfv$ to $M_\QQ$. 
Since $mv+\overline C$ agrees with $m\cdot\Box_h$ in an open neighborhood of $mv$, the
claim follows.

We now show part \ref{item1} of the proposition.
Let $\lambda_i$ be as in Lemma \ref{lemma:gorenstein}, and set $\lambda=\sum \lambda_i$. Note that $\lambda+1\leq m$, so $m\bfv+(\lambda+1) \bfw_i$ is a point on an edge of the graph of $(mh)^*_P$, and thus 
\[
mv+(\lambda+1) w_i\in m\cdot \Box_h.
\]
Since 
\[
	\frac{\lambda+1}{\lambda}u+mv=\sum \frac{\lambda_i}{\lambda} \left((\lambda+1)w_i+mv\right),
\]
and $\sum \lambda_i/\lambda=1$,
$\frac{\lambda+1}{\lambda}u+mv$ is in $m\cdot \Box_h$. By convexity, so is the line segment from $mv$ to $\frac{\lambda+1}{\lambda}u+mv$. Furthermore, since $\varepsilon u+mv$ is in the interior of $m\cdot \Box_h$, the interior of this segment is in the interior of $m\cdot \Box_h$. In particular, $u+mv$ is in the interior of $m\cdot \Box_h$, proving claim \ref{item1}.

We move on to part \ref{item2} of the proposition.
For each $i$ with $\lambda_i\neq 0$, let $S_i$  consist of those points  \[z\in \Inc(\Psi)_{mv+\lambda w_i}\] with 
$z_P=\Psi_P(mv+\lambda w_i)$. By construction, each $S_i$ is an $(r-2)$-simplex contained in a hyperplane of $\QQ_0^\cP$ parallel to the hyperplane $z_P=0$, where $r=\#\cP$.
We claim that $S_i$ has dilation strictly larger than $(r-2)$.
Indeed, 
\[
\deg \Psi=\deg (mh)^*+(r-2)\geq r-2
\]
so since $mv+(\lambda+1)w_i\in m\cdot \Box_h$, we have
\[\deg \Psi(mv+(\lambda+1) w_i)\geq r-2.\] Furthermore, since $\deg(mh)^*(mv)>0$ ($x$ was a contraction-free fixed point), it follows that $\deg \Psi(mv)>  r-2$. By the concavity of $\deg \Psi$, $\deg \Psi(mv+\lambda w_i)>  r-2$. But $\deg \Psi(mv+\lambda w_i)$ is the dilation of $S_i$, proving the claim.

Set \[S=\Inc(\Psi)_{u+mv}\cap \conv \{S_i\ |\ \lambda_i\neq 0\}.\]
Here, the convex hull is taken inside the space $M_\QQ\times \QQ^\cP_0$.
This is a non-empty simplex, since (by the linear independence of the $w_i$)
\[
	S=
% \Inc(\Psi)_{u+mv}\cap \conv \{S_i\ |\ \lambda_i\neq 0\}=
\sum\frac{\lambda_i}{\lambda} S_i.
\]
In fact, since each $S_i$ has dilation larger than $(r-2)$, so does $S$.
Furthermore, the $z_P$ coordinate of any point $z\in S$ is
\[
\sum \frac{\lambda_i}{\lambda}\Psi_P(mv+\lambda w_i)=\Psi_P(mv)+\sum\lambda_i\alpha_i=\Psi_P(mv)+\alpha\in\ZZ
\]   
so $S$ is an $(r-2)$-simplex in an integral hyperplane of $\QQ_0^\cP$ with dilation  larger than $(r-2)$. Hence, $S$ contains an interior lattice point, necessarily contained in the interior of $\Inc(\Psi)_{u+mv}$. This proves claim~\ref{item2}.
\end{proof}

\begin{cor}\label{cor:degree}
Let $m\geq d+1$.
The divisor 
\[
A=\sum_Q\left( \lceil(mh)_Q^*(u+mv)\rceil+b_Q-1\right) \cdot Q
\]
has degree at least $2g$.
\end{cor}
\begin{proof}
By Proposition \ref{prop:section}, there exists a lattice point $\bfu'$ in the interior of $\Inc(\Psi^*)$ with $p(\bfu')=u+mv$.
By Proposition \ref{prop:inc}, this means that there is a degree zero divisor $\sum c_P P$ on $Y$ such that 
\[
c_Q+\Psi_Q(u+mv)>0
\]
for all $Q\in \cP$,  that is,
\[
c_Q+\lceil(mh)_Q^*(u+mv)\rceil -1+b_Q'\geq 0.
\]
Thus, the divisor
\[
\sum \left(\lceil(mh)_Q^*(u+mv)\rceil -1+b_Q'\right)\cdot Q
\]
has non-negative degree.
Adding $E$, we obtain that
\begin{align*}
A=\sum_Q \left(\lceil(mh)_Q^*(u+mv)\rceil -1+b_Q'+e_Q\right)\cdot Q\\=\sum_Q \left(\lceil(mh)_Q^*(u+mv)\rceil+b_Q-1\right)\cdot Q
\end{align*}
has degree at least $2g$.
\end{proof}

We will now prove the freeness theorem: 
\begin{proof}[Proof of Theorem \ref{thm:main}]
Using notation established at the beginning of this section, we must show that no torus fixed point $x$ is in the base locus of $K_X+mH$, so long as $m\geq d+1$.
Since $X$ only has rational Gorenstein singularities by assumption, $K_X+mH$ is nef by \cite[Theorem 1]{fujita:87a}. But then by Proposition \ref{prop:bp}, no contraction fixed point is in the base locus of $K_X+mH$. Hence, in the following, we must only consider contraction-free fixed points. 

Let $x$ be a contraction-free fixed point of $X$; this corresponds to a vertex $\bfv$ of the graph of $h_P^*$.
By Corollary \ref{cor:degree}, the divisor 
$A=\sum_Q \left(\lceil(mh)_Q^*(u+mv)\rceil+b_Q-1\right)\cdot Q$ has degree at least $2g$, and hence is basepoint free on $Y$. In particular, $\CO(A)$  has a global section $f$ with
\[
\ord_P f+b_P+ \lceil(mh)_P^*(u+mv)\rceil-1=0.
\] 
This section $f$ satisfies
\[
\ord_Q f+(mh)_Q^*(u+mv)+b_Q>0
\]
for all $Q$, so by Proposition \ref{prop:canonical}, $f\chi^{u+mv}$ is a global section of $\CO(K_X+mH)$. 

On the other hand, using Lemma \ref{lemma:gorenstein}, a local generator for $\CO(K_X+mH)$ at $x$ has degree $u+mv$ and vanishing order $-\alpha-b_P-(mh)_P^*(mv)$ at $P$. But $\alpha$ is the largest integer such that $(u,\alpha)$ is contained in $C_\bfv^\circ$.
At the vertex $m\bfv$ of the graph of $(mh)_P^*$, the edges emanating from $m\bfv$ have primitive generators $\bfw_i$, and $m\bfw_i$ still lie on the graph. 
Since $m\geq \dim X$, it follows that $\alpha$ is the largest integer such that 
\[
\alpha< (mh)_P^*(u+mv)-(mh)_P^*(mv).
\]
Hence,  $f$ has vanishing order $-\alpha-b_P-(mh)_P^*(mv)$ at $P$, and thus $f\chi^{u+mv}$ is a local generator for $\CO(K_X+mH)$.
\end{proof}

\begin{remark}
	Two other natural ``combinatorial'' classes of varieties on which one could test Fujita's conjectures are \emph{Mori Dream Spaces} and \emph{projectivised toric vector bundles}. 	
	In \cite{fahrner}, Fahrner describes an algorithm for testing Fujita's freeness conjecture for any fixed Gorenstein Mori Dream Space $X$. Note that in fact, any rational complexity-one $T$-variety is automatically a Mori Dream Space.

	For $\cE$ an equivariant vector bundle on a toric variety $X$, sections of line bundles on $\PP(\cE)$ can also be described combinatorially, see e.g. \cite{parliaments}. This makes projectivised toric vector bundles into an attractive testing ground for Fujita's conjectures. Note that if $\cE$ has rank two, then $\PP(\cE)$ has the natural structure of a complexity-one $T$-variety. In particular, if $X$ is Gorenstein and $\cE$ has rank two, our Theorem \ref{thm:main} implies that Fujita's freeness conjecture holds for $\PP(\cE)$.
\end{remark}

\section{Non-free Ample Divisors}\label{sec:nonfree}
In this section, we give an example showing that on smooth $\bK^*$ surfaces, ample divisors can be ``arbitrarily far away'' from being basepoint free.

From \S\ref{sec:divisors} and \S\ref{sec:adjunction}, we saw that an ample  $T$-invariant divisor $D_h$ on $\tv(\dfan,\marking)$ determines piecewise affine concave functions \[h_P^*:\Box_h\to \QQ\]
whose graphs have integral vertices, such that $\deg h^* >0$ on $\Box_h^\circ$, and if $\deg h^*(v)=0$ at a vertex of $\Box_h$, then $h^*(v)$ is principal. Such a function $h^*$ is called a \emph{divisorial polytope} \cite[\S3]{polarized}. It turns out that we can reverse the above procedure: to any divisorial polytope we can associate a polarized complexity-one $T$-variety $(X,H)$ \cite[Theorem 3.2]{polarized}. We will use this construction to show the following:

\begin{thm}\label{thm:ex}
Let $k\in\NN$. There is a smooth rational projective $\KK^*$-surface $X_k$ with ample divisor $H_k$ such that for all $m\leq k$, $mH_k$ is not globally generated.
\end{thm}
\begin{proof}

For parameters $\alpha,\lambda\in \bN$, consider the following convex functions:

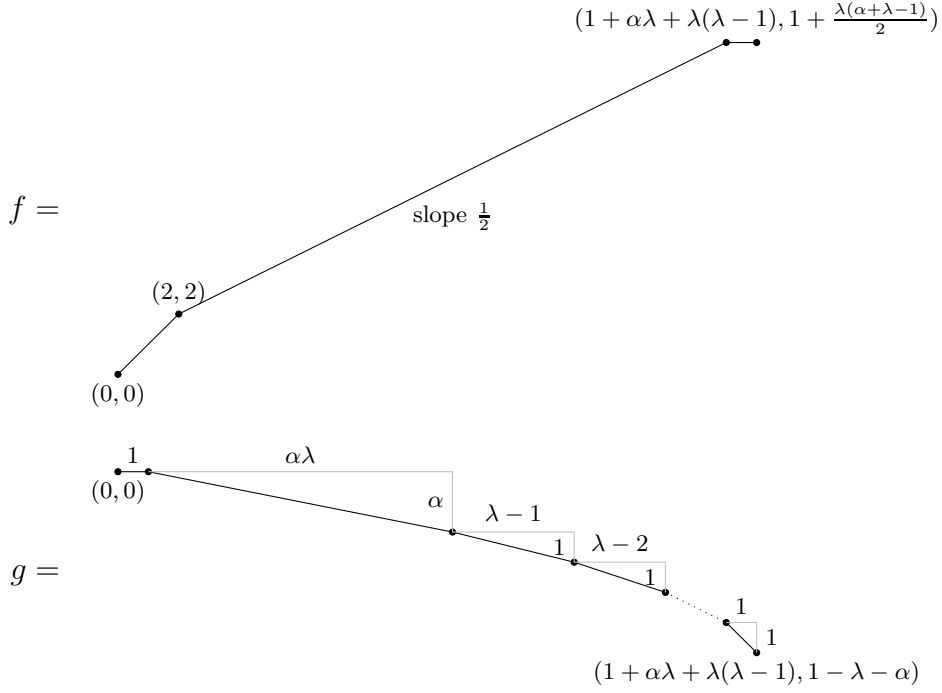
\begin{align*}
f&=
\begin{array}{c}
\begin{tikzpicture}[scale=.4]
\draw (0,0) -- (2,2) -- (20,11) -- (21,11);
\draw[fill] (0,0) circle [radius=.1];
\draw[fill] (2,2) circle [radius=.1];
\draw[fill] (20,11) circle [radius=.1];
\draw[fill] (21,11) circle [radius=.1];
\scriptsize{
\node [below] at (0,0) {$(0,0)$};
\node [above] at (2,2) {$(2,2)$};
\node [below] at (11,6) {slope $\frac{1}{2}$};
\node [above] at (21,11) {$(1+\alpha\lambda+\lambda(\lambda-1),1+\frac{\lambda(\alpha+\lambda-1)}{2})$};
}
\end{tikzpicture}
\end{array}\\
g&=
\begin{array}{c}
\begin{tikzpicture}[scale=.4]
\draw (0,0) -- (1,0) -- (11,-2) -- (15,-3) -- (18,-4);
\draw [dotted] (18,-4) -- (20,-5);
\draw (20,-5) -- (21,-6);
\draw[fill] (0,0) circle [radius=.1];
\draw[fill] (1,0) circle [radius=.1];
\draw[fill] (11,-2) circle [radius=.1];
\draw[fill] (15,-3) circle [radius=.1];
\draw[fill] (18,-4) circle [radius=.1];
\draw[fill] (20,-5) circle [radius=.1];
\draw[fill] (21,-6) circle [radius=.1];
\draw[lightgray] (1,0) -- (11,0) -- (11,-2) -- (15,-2)-- (15,-3)--(18,-3)--(18,-4);
\draw[lightgray] (20,-5)--(21,-5)--(21,-6);
\scriptsize{
\node [above] at (.5,0) {$1$};
\node [above] at (6,0) {$\alpha\lambda$};
\node [left] at (11,-1) {$\alpha$};
\node [above] at (13,-2) {$\lambda-1$};
\node [above] at (16.5,-3) {$\lambda-2$};
\node [above] at (20.5,-5) {$1$};
\node [left] at (15,-2.5) {$1$};
\node [left] at (18,-3.5) {$1$};
\node [right] at (21,-5.5) {$1$};
\node [below] at (0,0) {$(0,0)$};
\node [below] at (21,-6) {$(1+\alpha\lambda+\lambda(\lambda-1),1-\lambda-\alpha)$} ;
}
\end{tikzpicture}
\end{array}
\end{align*}

Then
\[
\Psi=f\otimes P_0+g\otimes(P_1+\ldots+P_\ell)
\]
for distinct points $P_0,\ldots,P_\ell\in \PP^1$ is a divisorial polytope
if
\[
1+\frac{\lambda(\alpha+\lambda-1)}{2}>\ell(\alpha+\lambda-1).
\]
This is certainly satisfied if $2\ell<\lambda$ and $\alpha$ is sufficiently large. 
The corresponding rational $\KK^*$-surface $X$ is smooth by \cite[Proposition 3.4]{polarized}.

However,
\[\deg \lfloor \Psi(2)\rfloor = 2+\ell\lfloor \frac{-1}{\lambda} \rfloor =2-\ell\]
which is negative for $\ell>2$ and $\lambda>1$. Hence, the corresponding ample divisor $H$ is not globally generated by Lemma \ref{lemma:bp}.
More generally, let $\Psi^m$ be the divisorial polytope corresponding to $mH$. It has a vertex at $u=2m$, and we have
\[\deg \lfloor \Psi^m(2m)\rfloor = 2m+\ell\lfloor \frac{-m}{\lambda} \rfloor\]
This is negative if $\ell>2m$ and $\lambda>m$. Thus, 
By choosing $\ell >2k$, $\lambda>2\ell$, and $\alpha$ sufficiently large, we obtain a polarized surface $(X_k,H_k)$ satisfying the hypotheses of the theorem.
Indeed, $\lambda>2\ell$ and $\alpha$ sufficiently large guarantee that the construction yields a divisorial polytope, while $\ell>2k$ will then guarantee the existence of a basepoint of $mH$.
\end{proof}

\bibliographystyle{alpha}
\bibliography{paper}

\end{document}